\documentclass[10pt]{amsart}

\usepackage[latin1]{inputenc}
\usepackage{amsmath}
\usepackage{amsfonts}
\usepackage{amssymb}
\usepackage[mathscr]{eucal}
\usepackage{diagrams}
\usepackage{xy,color,graphicx}
\usepackage{hyperref}
\xyoption{all}

\newcommand{\wis}[1]{{\text{\em \usefont{OT1}{cmtt}{m}{n} #1}}}

\newcommand{\C}{\mathbb{C}}
\newcommand{\Z}{\mathbb{Z}}
\newcommand{\vtx}[1]{*+[o][F-]{\scriptscriptstyle #1}}
\newcommand{\Oscr}{\mathcal{O}}

\newtheorem{theorem}{Theorem}
\newtheorem{lemma}{Lemma}

\title{Rationality and dense families of $B_3$ representations}
\author{Lieven Le Bruyn} 
\address{Department of Mathematics, University of Antwerp \\ 
 Middelheimlaan 1, B-2020 Antwerp (Belgium) \\ {\tt lieven.lebruyn@ua.ac.be}}

\begin{document}
\sloppy

\maketitle

\begin{abstract} Every irreducible component $\wis{iss}_{\beta}~\Gamma_0$ of semi-simple $n$-dimensional representations of the modular group $\Gamma_0 = PSL_2(\Z)$ has a Zariski dense subset contained in the image of an \'etale map
\[
\wis{iss}(Q,\alpha) \rTo \wis{iss}_{\beta}~\Gamma_0 \]
from the quotient variety $\wis{iss}(Q,\alpha)$ of representations of a fixed quiver $Q$ and a dimension vector $\alpha$ such that $\wis{iss}(Q,\alpha)$ is a rational variety. As an application we will prove that there is a unique component of 6-dimensional simple representations of the three string braid group $B_3$ detecting braid-reversion. Further, we give explicit rational parametrizations of dense families of simple $B_3$-representations of all dimensions $< 12$.
\end{abstract}

A knot is said to be invertible if it can be deformed continuously to itself, but with the orientation reversed. There exist non-invertible knots, the unique one with a minimal number of crossings is knot $8_{17}$  which is the closure of the three string braid $b=\sigma_1^{-2} \sigma_2 \sigma_1^{-1} \sigma_2 \sigma_1^{-1} \sigma_2^2$ 
\[
8_{17}~\includegraphics[height=70pt]{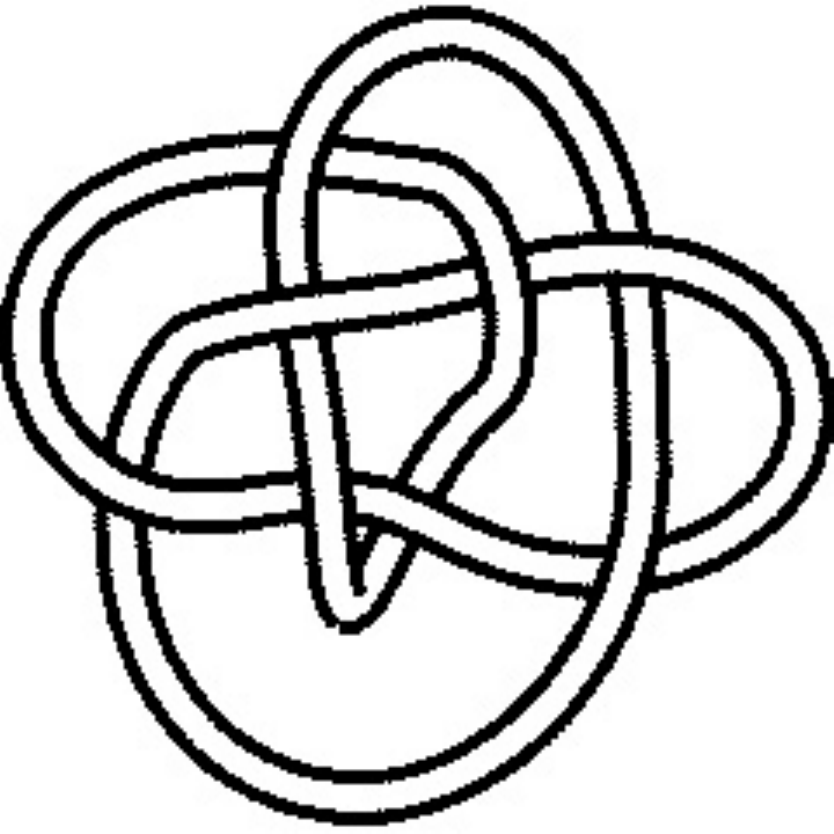}~\qquad~\qquad~\qquad~b~\includegraphics{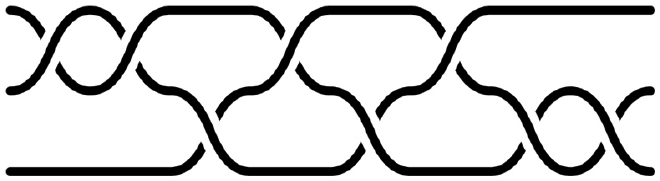}
\]
Proving knot-invertibility of $8_{17}$ essentially comes down to separating the conjugacy class of the braid $b$ from that of its reversed braid $b' = \sigma_2^2 \sigma_1^{-1} \sigma_2 \sigma_1^{-1} \sigma_2 \sigma_1^{-2}$. Recall that knot invariants derived from quantum groups cannot detect non-invertibility, see \cite{Kuperberg}. On the other hand, $Tr_V(b) \not= Tr_V(b')$ for a sufficiently large $B_3$-representation $V$. In fact, Bruce Westbury discovered such $12$-dimensional representations, and asked for the minimal dimension of a $B_3$-representation able to detect a braid from its reversed braid, \cite{Westbury2} .

Imre Tuba and Hans Wenzl have given a complete classification of all simple $B_3$-representations in dimension $\leq 5$, \cite{TubaWenzl}. By inspection, one verifies that none of these representation can detect invertibility, whence the minimal dimension must be $6$. Unfortunately, no complete classification is known of simple $B_3$-representations of dimension $\geq 6$. In this note, we propose a general method to solve this and similar separation problems for three string braids.

Let $\wis{rep}_n~B_3$ be the affine variety of all $n$-dimensional representations of the three string braid group $B_3$. There is a base change action of $GL_n$ on this variety having as its orbits the isomorphism classes of $n$-dimensional representations. The GIT-quotient of this action, that is, the variety classifying closed orbits
\[
\wis{rep}_n~B_3~//~GL_n = \wis{iss}_n~B_3 \]
is the affine variety $\wis{iss}_n~B_3$ whose points correspond to the isomorphism classes of semi-simple $n$-dimensional $B_3$-representations. In general, $\wis{iss}_n~B_3$ will have several irreducible components
\[
\wis{iss}_n~B_3 = \bigcup_{\alpha} \wis{iss}_{\alpha}~B_3 \]
If we can prove that $Tr_W(b_1) = Tr_W(b_2)$ for all representations $W$ in a Zariski-dense subset $Y_{\alpha} \subset \wis{iss}_{\alpha}~B_3$, then no representation $V$ in that component will be able to separate $b_1$ from $b_2$.
In order to facilitate the calculations, we would like to parametrize the dense family $Y_{\alpha}$ by a minimal number of free parameters. That is, if $dim~\wis{iss}_{\alpha}~B_3 = d$ we would like to construct explicitly a morphism $X_{\alpha} \rTo \wis{iss}_{\alpha}~B_3$ from a rational affine variety of dimension $d$, having a Zariski dense image in $\wis{iss}_{\alpha}~B_3$.

The theory of Luna-slices in geometric invariant theory, see \cite{Luna} and \cite{MumfordFogarty}, will provide us with a supply of affine varieties $X_{\alpha}$ and specific \'etale 'action'-maps $X_{\alpha} \rTo \wis{iss}_{\alpha}~B_3$. Rationality results on quiver representations, see \cite{BessLB} and \cite{Schofield}, will then allow us to prove rationality of some specific of these varieties $X_{\alpha}$. As the modular group $\Gamma_0 = B_3 / <c>$ is a central quotient of $B_3$ it suffices to obtain these results for $\Gamma_0$. In the first two sections we will prove

\begin{theorem} The affine variety classifying $n$-dimensional semi-simple representations of the modular group decomposes into a disjoint union of irreducible components
\[
\wis{iss}_n~\Gamma_0 = \bigsqcup_{\beta}~\wis{iss}_{\beta}~\Gamma_0 \]
There exists a fixed quiver $Q$ having the following property.  For every component $\wis{iss}_{\beta}~\Gamma_0$ containing a simple representation, there is a $Q$-dimension vector $\alpha$ with rational quotient variety $\wis{iss}(Q,\alpha)$ and an \'etale action map
\[
\wis{iss}(Q,\alpha) \rTo \wis{iss}_{\beta}~\Gamma_0 \]
having a Zariski dense image in $\wis{iss}_{\beta}~\Gamma_0$.
\end{theorem}

We apply this general method to solve Westbury's separation problem. Of the four irreducible components of $\wis{iss}_6~B_3$ the three of dimension 6 cannot detect invertibility, whereas the component of dimension 8 can. A specific representation in that component is given by the matrices
\[ 
\sigma_1 = 
 \begin{bmatrix}
\rho+1 & \rho -1 & \rho -1 & \rho-1 & -\rho+1 & -\rho+1 \\
-2 \rho -1 & -1 & -2\rho-1 & 2\rho+1 & -2 \rho-1 & 2\rho+1 \\
\rho+2 & \rho+2 & -\rho & \rho-2 & -\rho -2 & \rho+2 \\
-\rho-2 & -3 \rho & \rho+2 & -\rho+2 & 3 \rho &-\rho-2 \\
\rho-1 & -\rho+1 & 3 \rho + 3 & -\rho+1 & 3 \rho+1 & -3\rho -3 \\
-3 & -2\rho-1 & 2 \rho+1 & 3 & 2\rho+1 & -2 \rho-3
\end{bmatrix} \]
\[
\sigma_2 =
 \begin{bmatrix}
\rho +1 & \rho -1 & \rho -1 & -\rho+1 & \rho -1 & \rho -1 \\
-2\rho -1 & -1 & -2 \rho -1 & -2 \rho -1 & 2 \rho + 1 & -2\rho -1 \\
\rho+2 & \rho + 2 & -\rho & \rho+2 & \rho+2 & -\rho-2 \\
\rho+2 & 3\rho & -\rho-2 & \rho+2 & 3 \rho & -\rho -2 \\
-\rho + 1 & \rho -1 & -3 \rho -3 & -\rho+1 & 3 \rho + 1 & -3\rho -3 \\
3 & 2 \rho+1 & -2 \rho -1 & 3 & 2 \rho + 1 & -2\rho -3 
\end{bmatrix}
\]
where $\rho$ is a primitive third root of unity. One verifies that $Tr(b)=-7128 \rho -1092$ for the braid $b$ describing knot $8_{17}$, whereas $Tr(b')=7128 \rho +6036$ for the reversed braid $b'$.

In order to facilitate the application of his method  to other separation problems of three string braids, we provide in section three explicit rational parametrizations of dense families of simple $n$-dimensional $B_3$-representations for all components and all $n \leq 11$. This can be viewed as a first step towards extending the Tuba-Wenzl classification \cite{TubaWenzl}.

\section{Luna slices for representations of the modular group}

In this section we recall the reduction, due to Bruce Westbury \cite{Westbury}, of the study of finite dimensional simple representations of the three string braid group $B_3 = \langle \sigma_1,\sigma_2~|~\sigma_1 \sigma_2 \sigma_1 = \sigma_2 \sigma_1 \sigma_2 \rangle$ to those of a particular quiver $Q_0$. We will then identify the Luna slices at specific $Q_0$-representations to representation spaces of corresponding local quivers as introduced and studied in \cite{AdriLieven}. 

Recall that the center of $B_3$ is infinite cyclic with generator $c = (\sigma_1 \sigma_2)^3 = (\sigma_1 \sigma_2 \sigma_1)^2$ and hence that the corresponding quotient group (taking $S = \sigma_1 \sigma_2 \sigma_1$ and $T= \sigma_1 \sigma_2$)
\[
B_3/ \langle c \rangle = \langle \overline{S},\overline{T}~|~\overline{S}^2=\overline{T}^3=e \rangle \simeq C_2 \ast C_3 \]
is the free product of cyclic groups of order two and three and therefore isomorphic to the modular group $\Gamma_0 = PSL_2(\Z)$. 

By Schur's lemma, $c$ acts via scalar multiplication with $\lambda \in \C^*$ on any finite dimensional irreducible $B_3$-representation, hence it suffices to study the irreducible representations of the modular group $\Gamma_0$.
Bruce Westbury \cite{Westbury} established the following connection between irreducible representations of $\Gamma_0$ and specific stable representations of the directed quiver $Q_0$ :
\[
\xymatrix@=.4cm{
& & & & \vtx{} \\
\vtx{} \ar[rrrru] \ar[rrrrd] \ar[rrrrddd] & & & & \\
& & & & \vtx{} \\
\vtx{} \ar[rrrru] \ar[rrrruuu] \ar[rrrrd] & & & & \\
& & & &  \vtx{} }
\]
For $V$ an $n$-dimensional representation of $\Gamma_0$, decompose $V$ into eigenspaces with respect to the actions of $\overline{S}$ and $\overline{T}$
\[
V_{+} \oplus V_{-} = V = V_1 \oplus V_{\rho} \oplus V_{\rho^2} \]
($\rho$ a primitive 3rd root of unity). Denote the dimensions of these eigenspaces by $a=dim(V_+), b=dim(V_-)$ resp.  $x=dim(V_1), y=dim(V_{\rho})$ and $ z = dim(V_{\rho^2})$, then clearly $a+b=n=x+y+z$. 

Choose a vector-space basis for $V$ compatible with the decomposition $V_+ \oplus V_-$ and another basis  of $V$ compatible with the decomposition $V_1 \oplus V_{\rho} \oplus V_{\rho^2}$, then the associated base-change matrix $B \in GL_n(\C)$ determines a $Q_0$-representation $V_B$ of dimension vector $\alpha = (a,b;x,y,z)$
\begin{eqnarray} \label{identification}
\xymatrix@=.4cm{
& & & & \vtx{x} \\
\vtx{a} \ar[rrrru]^(.3){B_{11}} \ar[rrrrd]^(.3){B_{21}} \ar[rrrrddd]_(.2){B_{31}} & & & & \\
& & & & \vtx{y} \\
\vtx{b} \ar[rrrruuu]_(.7){B_{12}} \ar[rrrru]_(.7){B_{22}} \ar[rrrrd]_(.7){B_{23}} & & & & \\
& & & & \vtx{z}}~\qquad \text{with} \qquad B= \begin{bmatrix} B_{11} & B_{12} \\ B_{21} & B_{22} \\ B_{31} & B_{32} \end{bmatrix}
\end{eqnarray}
A $Q_0$-representation $W$ of dimension vector $\alpha$ is said to be $\theta$-stable, resp. $\theta$-semi-stable  if for every proper sub-representations $W'$, with dimension vector $\beta = (a',b';x;,y',z')$, we have that $x'+y'+z' > a'+b'$, resp. $x'+y'+z' \geq a'+b'$. 

\begin{theorem}[Westbury, \cite{Westbury}] $V$ is an $n$-dimensional simple $\Gamma_0$-representation if and only if the corresponding $Q_0$-representation $V_B$ is $\theta$-stable. Moreover, $V \simeq W$ as $\Gamma_0$-representations if and only if corresponding $Q_0$-representations $V_B$ and $W_{B'}$ are isomorphic as quiver-representations.

The affine GIT-quotient $\wis{iss}_n~\Gamma_0 = \wis{rep}_n~\Gamma_n / GL_n$ classifying isomorphism classes of $n$-dimensional semi-simple $\Gamma_0$-representations decomposes into a disjoint union of irreducible components
\[
\wis{iss}_n~\Gamma_0 = \bigsqcup_{\alpha}~\wis{iss}_{\alpha}~\Gamma_0 \]
one component for every dimension vector $\alpha = (a,b;x,y,z)$ satisfying $a+b=n=x+y+z$. 
If $\alpha=(a,b;x,y,z)$ satisfies $x.y.z \not= 0$, then the component $\wis{iss}_{\alpha}~\Gamma_0$ contains an open subset of simple representations if and only if $max(x,y,z) \leq min(a,b)$. In this case, the dimension of $\wis{iss}_{\alpha}~\Gamma_0$ is equal to $1+n^2-(a^2+b^2+x^2+y^2+z^2)$. 
\end{theorem}

The remaining simple $\Gamma_0$-representations (that is, those of dimension vector $\alpha = (a,b;x,y,z)$ with $x.y.z=0$) are of dimension one or two. There are $6$ one-dimensional simples (of the abelianization $\Gamma_{0,ab} = C_2 \times C_3$) corresponding to the $Q_0$-representations
\[
S_1 = \xymatrix@=.1cm{
& & & & \vtx{1} \\
\vtx{1} \ar[rrrru]^1  & & & & \\
& & & & \vtx{0} \\
\vtx{0}  & & & & \\
& & & & \vtx{0}} \qquad
S_2 = \xymatrix@=.1cm{
& & & & \vtx{0} \\
\vtx{0}  & & & & \\
& & & & \vtx{1} \\
\vtx{1}  \ar[rrrru]^1  & & & & \\
& & & & \vtx{0}} \qquad
S_3 = \xymatrix@=.1cm{
& & & & \vtx{0} \\
\vtx{1}  \ar[rrrrddd]^1 & & & & \\
& & & & \vtx{0} \\
\vtx{0}  & & & & \\
& & & & \vtx{1}} 
\]
\[
S_4 = \xymatrix@=.1cm{
& & & & \vtx{1} \\
\vtx{0}  & & & & \\
& & & & \vtx{0} \\
\vtx{1} \ar[rrrruuu]^1  & & & & \\
& & & & \vtx{0}} \qquad
S_5 = \xymatrix@=.1cm{
& & & & \vtx{0} \\
\vtx{1}  \ar[rrrrd]^1  & & & & \\
& & & & \vtx{1} \\
\vtx{0}  & & & & \\
& & & & \vtx{0}} \qquad
S_6 = \xymatrix@=.1cm{
& & & & \vtx{0} \\
\vtx{0}  & & & & \\
& & & & \vtx{0} \\
\vtx{1}  \ar[rrrrd]^1 & & & & \\
& & & & \vtx{1}} 
\]
and three one-parameter families of two-dimensional simple $\Gamma_0$-representations corresponding to the $Q_0$-representations
\[
T_1(\lambda) = \xymatrix@=.1cm{
& & & & \vtx{0} \\
\vtx{1}  \ar[rrrrd]|(0.6){\lambda} \ar[rrrrddd]|(0.7){1} & & & & \\
& & & & \vtx{1} \\
\vtx{1} \ar[rrrru]|(0.6){1}  \ar[rrrrd]|(0.6){1} & & & & \\
& & & &  \vtx{1} }
\quad
T_2(\lambda) = \xymatrix@=.1cm{
& & & & \vtx{1} \\
\vtx{1} \ar[rrrru]|(0.6){\lambda}  \ar[rrrrddd]|(0.6){1} & & & & \\
& & & & \vtx{0} \\
\vtx{1}  \ar[rrrruuu]|(0.6){1} \ar[rrrrd]|(0.6){1} & & & & \\
& & & &  \vtx{1} } \quad
T_3(\lambda) = \xymatrix@=.1cm{
& & & & \vtx{1} \\
\vtx{1} \ar[rrrru]|(0.6){\lambda} \ar[rrrrd]|(0.6){1}  & & & & \\
& & & & \vtx{1} \\
\vtx{1} \ar[rrrru]|(0.6){1} \ar[rrrruuu]|(0.7){1}  & & & & \\
& & & &  \vtx{0} }
\]
satisfying $\lambda \not= 1$. With $\mathcal{S}_1$ we will denote the set $\{ S_1,S_2,S_3,S_4,S_5,S_6 \}$ of all one-dimensional $\Gamma_0$-representations.

Consider a finite set $\mathcal{S}= \{ V_1,\hdots,V_k \}$ of simple $\Gamma_0$-representations and identify $V_i$ with the corresponding $Q_0$-representation of dimension vector $\alpha_i$. Consider the semi-simple $\Gamma_0$-representation
\[
M = V_1^{\oplus m_1} \oplus \hdots \oplus V_k^{\oplus m_k} \]
The theory of Luna slices allows us to describe the \'etale local structure of the component $\wis{iss}_{\beta}~\Gamma_0$, where $\beta = \sum_i m_i \alpha_i$, in a neighborhood of the point corresponding to $M$. We will assume throughout that $\beta = (a,b;x,y,z)$ is the dimension vector of a $\theta$-stable representation, that is, that $max(x,y,z) \leq min(a,b)$.

Let $\Oscr(M)$ be the $GL(\beta) = GL_a \times GL_b \times GL_x \times GL_y \times GL_z$-orbit of $M$ in the representation space $\wis{rep}(Q_0,\beta)$, then the normal space to the orbit
\[
N_M = \frac{T_M(\wis{rep}(Q_0,\beta))}{T_M(\Oscr(M))} \simeq Ext^1_{Q_0}(M,M) \]
is the extension space, see for example \cite[II.2.7]{Kraftbook}. Because
\[
Ext^1_{Q_0}(M,M) = \begin{bmatrix} M_{m_1}(Ext^1_{Q_0}(V_1,V_1)) & \hdots & M_{m_1 \times m_k}(Ext^1_{Q_0}(V_1,V_k)) \\
\vdots & & \vdots \\
M_{m_k \times m_1}(Ext^1_{Q_0}(V_k,V_1)) & \hdots & M_{m_k}(Ext^1_{Q_0}(V_k,V_k)) \end{bmatrix} \]
we can identify the vectorspace $Ext^1_{Q_0}(M,M)$ to the representation space $\wis{rep}(Q_{\mathcal{S}},\alpha_M)$ of the quiver $Q_{\mathcal{S}}$ on $k$ vertices $\{ v_1,\hdots,v_k \}$ (vertex $v_i$ corresponding to the simple $\Gamma_0$-representation $V_i$) such that the number of directed arrows from vertex $v_i$ to vertex $v_j$ is equal to
\[
\# \{ \xymatrix{\vtx{i} \ar[r] & \vtx{j}} \} = dim_{\C}~Ext^1_{Q_0}(V_i,V_j) \]
and the dimension vector $\alpha_M$ of $Q_{\mathcal{S}}$ is given by the multiplicities of the simple factors in $M$, that is, $\alpha_M = (m_1,\hdots,m_k)$. Observe that the stabilizer subgroup of $M$ is equal to $GL(\alpha_M) = GL_{m_1} \times \hdots \times GL_{m_k}$.

The quiver $Q_{\mathcal{S}}$ is called the local quiver of $M$, see for example \cite{AdriLieven} or \cite{LBbook}, and can be determined from the Euler form of $Q_0$ which is the bilinear map
\[
\chi_{Q_0}~:~\Z^5 \times \Z^5 \rTo \Z~\qquad \chi_{Q_0}(\alpha,\beta) = \alpha.M_{Q_0}.\beta^{tr} \]
determined by the matrix
\[
M_{Q_0} = \begin{bmatrix} 1 & 0 & -1 & -1 & -1 \\ 0 & 1 & -1 & -1 & -1 \\ 0 & 0 & 1 & 0 & 0 \\ 0 & 0 & 0 & 1 & 0 \\ 0 & 0 & 0 & 0 & 1 \end{bmatrix}  \]
For $V$ and $W$ $Q_0$-representation of dimension vector $\alpha$ and $\beta$, we have
\[
dim_{\C}~Hom_{Q_0}(V,W) - dim_{\C}~Ext^1_{Q_0}(V,W) = \chi_{Q_0}(\alpha,\beta) \]
Recall that $Hom_{Q_0}(V,W) = \delta_{VW} \C$ whenever $V$ and $W$ are $\theta$-stable quiver representations. From the Luna slice theorem we obtain, see for example \cite[\S 4.2]{LBbook}.

\begin{theorem} \label{general} Let $\mathcal{S} = \{ V_1,\hdots,V_k \}$ be a finite set of simple $\Gamma_0$-representations with corresponding $Q_0$-dimension vectors $\alpha_i$. Consider the semi-simple $\Gamma_0$-representation
\[
M = V_1^{\oplus m_1} \oplus \hdots \oplus V_k^{\oplus m_k} \]
with $Q_0$-dimension vector $\beta = \sum_i m_i \alpha_i$. Let $Q_{\mathcal{S}}$ be the local quiver described above and let $\alpha_M=(m_1,\hdots,m_k)$ be the $Q_{\mathcal{S}}$-dimension vector determined by the multiplicities. Then, the action map 
\[
GL(\beta) \times^{GL(\alpha_M)} \wis{rep}(Q_{\mathcal{S}},\alpha_M) \rTo \wis{rep}_{\beta}~\Gamma_0 \]
sending the class of $(g,N)$ in the associated fibre bundle to the representation $g.(M+N)$ where $M+N$ is the representation in the normal space to the orbit $\Oscr(M)$ corresponding to the $Q_{\mathcal{S}}$-representation $N$, is a $GL(\beta)$-equivariant \'etale map with a Zariski dense image. Taking $GL(\beta)$ quotients on both sides, we obtain an \'etale action map
\[
\wis{iss}(Q_{\mathcal{S}},\alpha_M) \rTo \wis{iss}_{\beta}~\Gamma_0 \]
with a Zariski dense image.
\end{theorem}

\section{The action maps for $\mathcal{S}_1$ and rationality}

In order to apply theorem~\ref{general} we need to consider a family $\mathcal{S}$ of simple $\Gamma_0$-representations generating a semi-simple representation $M$ in every component $\wis{iss}_{\beta}~\Gamma_0$, and, we need to make the action map explicit, that is, we need to identify representations of the quiver $Q_{\mathcal{S}}$ with representations in the normal space to the orbit $\Oscr(M)$.

We consider the set $\mathcal{S}_1 = \{ S_1, S_2, S_3, S_4, S_5, S_6 \}$ of all one-dimensional $\Gamma_0$-representations, using the notations as before. Consider the semi-simple $\Gamma_0$-representation of dimension $n=\sum_i a_i$
\[
M = S_1^{\oplus a_1} \oplus S_2^{\oplus a_2} \oplus S_3^{\oplus a_3} \oplus S_4^{\oplus a_4} \oplus S_5^{\oplus a_5} \oplus S_6^{\oplus a_6} \]
Then $M$ corresponds to a $\theta$-semistable $Q_0$-representation of dimension vector
\[
\beta_M = (a_1+a_3+a_5,a_2+a_4+a_6,a_1+a_4,a_2+a_5,a_3+a_6) \]
Under the identification (\ref{identification}) the $n \times n$ matrix $B = (B_{ij})_{ij}$ determined by $M$ consists of the following block-matrices $B_{ij}$ containing themselves blocks of sizes $a_u \times a_v$ for the appropriate $u$ and $v$
\[
B_{11} = \begin{bmatrix} 1_{a_1} & 0 & 0 \\ 0 & 0 & 0 \end{bmatrix} \qquad
B_{21} = \begin{bmatrix} 0 & 0 & 0 \\ 0 & 0 & 1_{a_5} \end{bmatrix} \qquad
B_{31} = \begin{bmatrix} 0 & 1_{a_3} & 0 \\ 0 & 0 & 0 \end{bmatrix}
\]
\[
B_{12} = \begin{bmatrix} 0 & 0 & 0 \\ 0 & 1_{a_4} & 0 \end{bmatrix} \qquad
B_{22} = \begin{bmatrix} 1_{a_2} & 0 & 0 \\ 0 & 0 & 0 \end{bmatrix} \qquad
B_{32} =\begin{bmatrix} 0 & 0 & 0 \\ 0 & 0 & 1_{a_6} \end{bmatrix} \]

\begin{theorem} \label{matrices} With $\alpha_M = (a_1,\hdots,a_6)$, the \'etale action map 
\[
\wis{iss}(Q_{\mathcal{S}_1},\alpha_M) \rTo \wis{iss}_{\beta_M}~\Gamma_0
\]
is induced by sending a representation in $\wis{rep}(Q_{\mathcal{S}_1},\alpha_M)$ defined by the matrices
\[
\xymatrix@=1.1cm{
& \vtx{a_1} \ar@/^/[ld]^{C_{16}} \ar@/^/[rd]^{C_{12}} & \\
\vtx{a_6} \ar@/^/[ru]^{C_{61}}  \ar@/^/[d]^{C_{65}} & & \vtx{a_2} \ar@/^/[lu]^{C_{21}} \ar@/^/[d]^{C_{23}} \\
\vtx{a_5} \ar@/^/[u]^{C_{56}}  \ar@/^/[rd]^{C_{54}} & & \vtx{a_3} \ar@/^/[u]^{C_{32}} \ar@/^/[ld]^{C_{34}} \\
& \vtx{a_4} \ar@/^/[lu]^{C_{45}} \ar@/^/[ru]^{C_{43}}  & }
\]
to the representation in $\wis{rep}(Q_0,\beta_M)$ corresponding to the $n \times n$ matrix (via identification (\ref{identification}))
\[
B = \begin{bmatrix} 1_{a_1} & 0 & 0 & C_{21} & 0 & C_{61} \\
0 & C_{34} & C_{54} & 0 & 1_{a_4} & 0 \\
C_{12} & C_{32} & 0 & 1_{a_2} & 0 & 0 \\
0 & 0 & 1_{a_5} & 0 & C_{45} & C_{65} \\
0 & 1_{a_3} & 0 & C_{23} & C_{43} & 0 \\
C_{16} & 0 & C_{56} & 0 & 0 & 1_{a_6} \end{bmatrix} 
\]
Under this map, simple $Q_{\mathcal{S}_1}$-representations with invertible matrix $B$ are mapped to irreducible $n$-dimensional $\Gamma_0$-representations. 

Hence, if the coefficients in the matrices $C_{ij}$ give a parametrization of (an open set of) the quotient variety $\wis{iss}(Q_{\mathcal{S}_1},\alpha_M)$, the $n$-dimensional representations of the three string braid group $B_3$ given by
\[
\begin{cases}
\sigma_1 \mapsto \lambda B^{-1} \begin{bmatrix} 1_{a_1+a_4} & 0 & 0 \\ 0 & \rho^2 1_{a_2+a_5} & 0 \\ 0 & 0 & \rho 1_{a_3+a_6} \end{bmatrix} B \begin{bmatrix} 1_{a_1+a_3+a_5} & 0 \\ 0 & -1_{a_2+a_4+a_6} \end{bmatrix} \\
\sigma_2 \mapsto \lambda \begin{bmatrix} 1_{a_1+a_3+a_5} & 0 \\ 0 & -1_{a_2+a_4+a_6} \end{bmatrix} B^{-1}  \begin{bmatrix} 1_{a_1+a_4} & 0 & 0 \\ 0 & \rho^2 1_{a_2+a_5} & 0 \\ 0 & 0 & \rho 1_{a_3+a_6} \end{bmatrix} B
\end{cases}
\]
contain a Zariski dense set of the simple $B_3$-representations in $\wis{iss}_{\beta_M}~B_3$.
\end{theorem}

\begin{proof}
Using the Euler-form of the quiver $Q_0$ and the $Q_0$-dimension vectors of the representations $S_i$ one verifies that the quiver $Q_{\mathcal{S}_1}$ is the one given in the statement of the theorem.
To compute the components of the tangent space in $M$ to the orbit, take $Lie(GL(\beta_M))$ as the set of matrices (in block-matrices of sizes $a_u \times a_v$)
\[
 \begin{bmatrix} A_1 & A_{13} & A_{15} \\
A_{31} & A_3 & A_{35} \\
A_{51} & A_{53} & A_5 \end{bmatrix} \oplus \begin{bmatrix} A_2 & A_{24} & A_{26} \\ A_{42} & A_4 & A_{46} \\ A_{62} & A_{64} & A_6 \end{bmatrix} \oplus \begin{bmatrix} A'_1 & A_{14} \\ A_{41} & A'_4 \end{bmatrix} \oplus \begin{bmatrix} A'_2 & A_{25} \\ A_{52} & A'_5 \end{bmatrix} \oplus
\begin{bmatrix} A'_3 & A_{36} \\ A_{63} & A'_6 \end{bmatrix} \]
and hence the tangent space to the orbit is computed using the action of $GL(\beta_M)$ on the quiver-representations, giving for example for the $B_{11}$-arrow
\[
(\begin{bmatrix} 1_{a_1} & 0 \\ 0 & 1_{a_4} \end{bmatrix} + \epsilon \begin{bmatrix} A'_1 & A_{14} \\ A_{41} & A'_4 \end{bmatrix}).\begin{bmatrix} 1_{a_1} & 0 & 0 \\ 0 & 0 & 0 \end{bmatrix}.(\begin{bmatrix}
1_{a_1} & 0 & 0 \\ 0 & 1_{a_3} & 0 \\ 0 & 0 & 1_{a_5} \end{bmatrix} - \epsilon \begin{bmatrix}
A_1 & A_{13} & A_{15} \\ A_{31} & A_3 & A_{35} \\ A_{51} & A_{53} & A_5 \end{bmatrix}) \]
which is equal to
\[
\begin{bmatrix} 1_{a_1} & 0 & 0 \\ 0 & 0 & 0 \end{bmatrix} + \epsilon \begin{bmatrix}
A_1-A'_1 & -A_{13} & -A_{15} \\ A_{41} & 0 & 0 \end{bmatrix} \]
and, similarly,  the $\epsilon$-components of $B_{21},B_{31},B_{12},B_{22}$ resp. $B_{23}$ are calculated to be
\[
B_{21}  :  \begin{bmatrix} 0 & 0 & A_{25} \\ -A_{51} & -A_{53} & A_5-A'_5 \end{bmatrix} \qquad B_{31} : \begin{bmatrix} -A_{31} & A_3-A'_3 & -A_{35} \\ 0 & A_{63} & 0 \end{bmatrix} \]
\[
B_{12} : \begin{bmatrix} 0 & A_{41} & 0 \\ -A_{42} & A_4-A'_4 & -A_{46}  \end{bmatrix} \qquad
B_{22} : \begin{bmatrix} A_2-A'_2 & -A_{24} & -A_{26} \\ A_{52} & 0 & 0 \end{bmatrix} \]
\[
B_{32} : \begin{bmatrix} 0 & 0 & A_{36} \\ -A_{62} & -A_{64} & A_6-A'_6 \end{bmatrix}
\]
Here the zero blocks correspond precisely to the matrices $C_{ij}$ describing a representation in $\wis{rep}(Q_{\mathcal{S}_1},\alpha_M)$ which can therefore be identified with the normal space in $M$ to the orbit $\Oscr(M)$.
Here we use the inproduct on $T_M~\wis{rep}_{\beta_M}~\Gamma_0 = \wis{rep}_{\beta_M}~\Gamma_0$ defined for all $B=(B_{ij})$ and $B'=(B_{ij}')$
\[
\langle B,B' \rangle = \sum_{ij} B_{ij} \overline{B}_{ij}^{tr} \]
 Hence, the representation $M + N$ is determined by the matrices
\[
B_{11} = \begin{bmatrix} 1_{a_1} & 0 & 0 \\ 0 & C_{34} & C_{54} \end{bmatrix} \qquad B_{21} =
\begin{bmatrix} C_{12} & C_{32} & 0 \\ 0 & 0 & 1_{a_5} \end{bmatrix} \]
\[
B_{31} = \begin{bmatrix} 0 & 1_{a_3} & 0 \\ C_{16} & 0 & C_{56} \end{bmatrix} \qquad
B_{12} = \begin{bmatrix}  C_{21} & 0 & C_{61} \\ 0 & 1_{a_4} & 0 \end{bmatrix}
\]
\[
B_{22} = \begin{bmatrix} 1_{a_2} & 0 & 0 \\ 0 & C_{45} & C_{65} \end{bmatrix} \qquad
B_{32} = \begin{bmatrix} C_{23} & C_{43} & 0 \\ 0 & 0 & 1_{a_6} \end{bmatrix}
\]
To any $Q_0$-representation $V_B$ of dimension vector $\alpha=(a,b;x,y,z)$, with invertible $n \times n$ matrix $B$ corresponds, via the identifications given in the previous section, the $n$-dimensional representation of the modular group $\Gamma_0$ defined by
\[
\begin{cases}
\overline{S} \mapsto \begin{bmatrix} 1_a & 0 \\ 0 & -1_b \end{bmatrix} \\
\overline{T} \mapsto B^{-1} \begin{bmatrix} 1_x & 0 & 0 \\ 0 & \rho 1_y & 0 \\ 0 & 0 & \rho^2 1_z \end{bmatrix} B
\end{cases}
\]
As $S=\sigma_1 \sigma_2 \sigma_1$ and $T=\sigma_1 \sigma_2$ it follows that $\sigma_1 = T^{-1} S$ and $\sigma_2 = S T^{-1}$. Therefore, when $V_B$ is $\theta$-stable it determines a $\C^*$-family of $n$-dimensional representations of the three string braid group $B_3$ given by
\[
\begin{cases}
\sigma_1 \mapsto \lambda B^{-1} \begin{bmatrix} 1_x & 0 & 0 \\ 0 & \rho^2 1_y & 0 \\ 0 & 0 & \rho 1_z \end{bmatrix} B \begin{bmatrix} 1_a & 0 \\ 0 & -1_b \end{bmatrix} \\
\sigma_2 \mapsto \lambda \begin{bmatrix} 1_a & 0 \\ 0 & -1_b \end{bmatrix} B^{-1}  \begin{bmatrix} 1_x & 0 & 0 \\ 0 & \rho^2 1_y & 0 \\ 0 & 0 & \rho 1_z \end{bmatrix} B
\end{cases}
\]
Using the dimension vector $\beta_M$ and the matrices $B_{ij}$ found before, we obtain the required $B_3$-representations. The final statements follow from theorem~\ref{general}.
\end{proof}

The method of proof indicates how one can make the action maps explicit for any given finite family of irreducible $\Gamma_0$-representations. Note that the condition for $\beta_M$ to be the dimension vector of a $\theta$-stable representation is equivalent to the condition on $\alpha_M$  
\[
a_i \leq a_{i-1} + a_{i+1} \qquad \text{for all $i~mod~6$} \]
which is the condition for $\alpha_M = (a_1,\hdots,a_6)$ to be the dimension vector of a simple $Q_{\mathcal{S}_1}$-representation, by \cite{LBProcesi}.

\begin{theorem} \label{rationality} For any $Q_0$-dimension vector $\beta = (a,b;x,y,z)$ admitting a $\theta$-stable representation, that is satisfying  $max(x,y,z) \leq min(a,b)$ there exist semi-simple representations $M \in \wis{iss}_{\beta}~\Gamma_0$
\[
M_{\beta} = S_1^{\oplus a_1} \oplus S_2^{\oplus a2} \oplus S_3^{\oplus a_3} \oplus S_4^{\oplus a_4} \oplus S_5^{\oplus a_5} \oplus S_6^{\oplus a_6} \]
such that for $\alpha_M = (a_1,\hdots,a_6)$ the quotient variety $\wis{iss}(Q_{\mathcal{S}_1},\alpha_M)$ is rational. As a consequence, the \'etale action map
\[
\wis{iss}(Q_{\mathcal{S}_1},\alpha_M) \rTo \wis{iss}_{\beta}~\Gamma_0 \]
given by theorem~\ref{matrices} determines a rational parametrization of a Zariski dense subset of $\wis{iss}_{\beta}~B_3$.
\end{theorem}

\begin{proof} 
The condition on $\beta$ ensures that there are simple $\alpha_M$-dimensional representations of $Q_{\mathcal{S}_1}$, and hence that $\alpha_M$ is a Schur root. By a result of Aidan Schofield \cite{Schofield} this implies that the quotient variety $\wis{iss}(Q_{\mathcal{S}_1},\alpha_M)$ is birational to the quotient variety of $p$-tuples of $h \times h$ matrices under simultaneous conjugation, where
\[
h = gcd(a_1,\hdots,a_l) \quad \text{and} \quad p=1-\chi_{Q_{\mathcal{S}_1}}(\frac{\alpha_M}{h},\frac{\alpha_M}{h}) = 1 + \frac{1}{h^2}(2 \sum_{i=1}^6 a_ia_{i+1} - \sum_{i=1}^6 a_i^2)  \]
By Procesi's result \cite[Prop. IV.6.4]{ProcesiBook} and the known rationality results (see for example \cite{BessLB}) the result follows when we can find such an $\alpha_M$ satisfying $gcd(a_1,\hdots,a_6) \leq 4$.
For small dimensions one verifies this by hand, and, for larger dimensions having found a representation $M$ with $gcd(a_1,\hdots,a_6) > 4$, one can modify the multiplicities of the simple components by 
\[
\xymatrix@=.6cm{
& 2  & \\
1  & & 1   \\
-1 & & -1 \\
& -2 &
}~\qquad \text{or} \qquad
\xymatrix@=.6cm{
& 3  & \\
1   & & 2   \\
-2 & & -1 \\
& -3 &
}
\]
(or a cyclic permutation of these) to obtain another semi-simple representation $M'$ lying in the same component, but having $gcd(\alpha_{M'}) =1$.
\end{proof}

\section{Rational dense families for $n < 12$}

In this section we will determine for every irreducible component $\wis{iss}_{\beta}~\Gamma_0$, with $\beta=(a,b;x,y,z)$ such that $a+b=x+y+z=n < 12$ satisfying $max(x,y,z) \leq min(a,b)$, a semi-simple $\Gamma_0$-representation
\[
M_{\beta} = S_1^{\oplus a_1} \oplus S_2^{\oplus a2} \oplus S_3^{\oplus a_3} \oplus S_4^{\oplus a_4} \oplus S_5^{\oplus a_5} \oplus S_6^{\oplus a_6} \]
contained in $\wis{iss}_{\beta}~\Gamma_0$. Then, we will explicitly determine a rational family of representations in $\wis{rep}(Q_{\mathcal{S}_1},\alpha_M)$ for $\alpha_M=(a_1,\hdots,a_6)$.

As we are only interested in the corresponding $B_3$-representations we may assume that $\beta=(a,b;x,y,z)$ is such that $a \geq b$ and $x \geq y \geq z$. Observe that in going from $\Gamma_0$-representations to $B_3$-representations we multiply with $\lambda \in \C^*$. Hence, there is a $\mu_6$-action on the dimension vectors of $Q_0$-representations giving the same component of $B_3$-representations. That is, we have to consider only one dimension vector from the $\mu_6$-orbit
\[
\begin{diagram}
(a,b;x,y,z) & \rTo & (b,a;z,x,y) & \rTo & (a,b;y,z,x)  \\
\uTo & & & & \dTo \\
(b,a;y,z,x)  & \lTo & (a,b;z,x,y) & \lTo & (b,a;x,y,z) 
\end{diagram}
\]
Hence we may assume that $a \geq b$ and that $x=max(x,y,z)$. If we find a module $M_{\beta}$ and corresponding rational parametrization for $\beta=(a,b;x,y,z)$ by representations in $\wis{rep}(Q_{\mathcal{S}_1},\alpha_M)$, then we can mirror that description to obtain similar data for $\beta'=(a,b;x,z,y)$ via
\[
\xymatrix@=1.1cm{
& \vtx{a_1} \ar@/^/[ld]^{C_{16}} \ar@/^/[rd]^{C_{12}} & \\
\vtx{a_6} \ar@/^/[ru]^{C_{61}}  \ar@/^/[d]^{C_{65}} & & \vtx{a_2} \ar@/^/[lu]^{C_{21}} \ar@/^/[d]^{C_{23}} \\
\vtx{a_5} \ar@/^/[u]^{C_{56}}  \ar@/^/[rd]^{C_{54}} & & \vtx{a_3} \ar@/^/[u]^{C_{32}} \ar@/^/[ld]^{C_{34}} \\
& \vtx{a_4} \ar@/^/[lu]^{C_{45}} \ar@/^/[ru]^{C_{43}}  & }~\quad \leftrightarrow \quad
\xymatrix@=1.1cm{
& \vtx{a_1} \ar@/^/[ld]^{C_{12}} \ar@/^/[rd]^{C_{16}} & \\
\vtx{a_2} \ar@/^/[ru]^{C_{21}}  \ar@/^/[d]^{C_{23}} & & \vtx{a_6} \ar@/^/[lu]^{C_{61}} \ar@/^/[d]^{C_{65}} \\
\vtx{a_3} \ar@/^/[u]^{C_{32}}  \ar@/^/[rd]^{C_{34}} & & \vtx{a_5} \ar@/^/[u]^{C_{56}} \ar@/^/[ld]^{C_{54}} \\
& \vtx{a_4} \ar@/^/[lu]^{C_{43}} \ar@/^/[ru]^{C_{45}}  & }
\]
proving that we may indeed restrict to $\beta=(a,b;x,y,z)$ with $a \geq b$ and $x \geq y \geq z$. 

\begin{lemma} \label{lem1} The following representations provide a rational family determining a dense subset of the quotient varieties $\wis{iss}(Q_{\mathcal{S}_1},\alpha)$ for these (minimalistic) dimension vectors $\alpha$.
\[
\begin{array}{l l}
(A) :~$\xymatrix@=.5cm{\vtx{1} \ar@/^/[rr]^1 & & \vtx{1} \ar@/^/[ll]^a}$ & 
(B) :~$\xymatrix@=.5cm{\vtx{1} \ar@/^/[rr]^{\begin{bmatrix} 0 \\ 1 \end{bmatrix}} & & \vtx{2} \ar@/^/[ll]^{\begin{bmatrix} 1 & a \end{bmatrix}} \ar@/^/[rr]^{\begin{bmatrix} b & c \end{bmatrix}} & & \vtx{1} \ar@/^/[ll]^{\begin{bmatrix} 1 \\ 0 \end{bmatrix}}}$ \\
(C) :~$\xymatrix@=.5cm{\vtx{1} \ar@/^/[rr]^{\begin{bmatrix} 0 \\ 1 \end{bmatrix}} & & \vtx{2} \ar@/^/[ll]^{\begin{bmatrix} 1 & a \end{bmatrix}} \ar@/^/[rr]^{\begin{bmatrix} 1 & 0 \\ 0 & 1 \end{bmatrix}} & & \vtx{2} \ar@/^/[ll]^{\begin{bmatrix} b & c \\ d & 1 \end{bmatrix}}}$ & 
(D) :~$\xymatrix@=.5cm{\vtx{2} \ar@/^/[rr]^{\begin{bmatrix} a & 0 \\ 0 & b \end{bmatrix}} & & \vtx{2} \ar@/^/[ll]^{\begin{bmatrix} 1 & 0 \\ 0 & 1  \end{bmatrix}} \ar@/^/[rr]^{\begin{bmatrix} 1 & 0 \\ 0 & 1 \end{bmatrix}} & & \vtx{2} \ar@/^/[ll]^{\begin{bmatrix} c & 1 \\ d & e \end{bmatrix}}}$ \\
(E) :~$\xymatrix@=.5cm{\vtx{1} \ar@/^/[rr]^{\begin{bmatrix} 1 \\ 0 \\ 0 \end{bmatrix}} & & \vtx{3} \ar@/^/[ll]^{\begin{bmatrix} a & 1 & b  \end{bmatrix}} \ar@/^/[rr]^{\begin{bmatrix} 1 & c & d \\ 0 & e & 1 \end{bmatrix}} & & \vtx{2} \ar@/^/[ll]^{\begin{bmatrix} 0 & 0 \\ 1 & 0 \\ 0 & 1 \end{bmatrix}}}$ & \\
\end{array}
\]
\end{lemma}

\begin{proof} Recall from \cite{LBProcesi} that the rings of polynomial quiver invariants are generated by traces along oriented cycles in the quiver. It also follows from \cite{LBProcesi} that these dimension vectors are the dimension vectors of simple representations and that their quotient varieties are rational of transcendence degrees : 1(A), 3(B), 4(C), 5(D) and 5(E). This proves (A). Cases (B),(C) and (D) are easily seen (focus on the middle vertex) to be equivalent to the problem of classifying couples of $2 \times 2$ matrices $(A,B)$ up to simultaneous conjugation in case (D). For (C) the matrix $A$ needs to have rank one and for (B) both $A$ and $B$ have rank one. It is classical that the corresponding rings of polynomial invariants are : (B)~$\C[Tr(A),Tr(B),Tr(AB)]$, (C)~$\C[Tr(A),Tr(B),Det(B),Tr(AB)]$ and (D)~$\C[Tr(A),Det(A),Tr(B),Det(B),Tr(AB)]$. In case (E), as $3 \geq 1,2$ we can invoke the first fundamental theorem for $GL_n$-invariants (see \cite[Thm. II.4.1]{Kraftbook}) to eliminate the $GL_3$-action by composing arrows through the $3$-vertex. This reduces the study to the quiver-representations below, of which the indicated representations form a rational family by case (D)
\[
\xymatrix@=.5cm{ \vtx{1} \ar@(ul,dl)_a \ar@/^/[rr]^{\begin{bmatrix} 1 \\ 0 \end{bmatrix}} & & \vtx{2} \ar@(ur,dr)^{\begin{bmatrix} c & d \\ e & 1 \end{bmatrix}} \ar@/^/[ll]^{\begin{bmatrix} 1 & b \end{bmatrix}}} \]
and one verifies that the indicated family of (E) representations does lead to these matrices.
\end{proof}

\begin{lemma} \label{lem2} Adding one vertex  at a time (until one has at most $5$ vertices) to either end of the full subquivers of $Q_{\mathcal{S}_1}$ of lemma~\ref{lem1} one can extend the dimension vector and the rational dense family of representations as follows in these allowed cases
\begin{itemize}
\item{$\xymatrix@=.5cm{\vtx{1} \ar@/^/[rr]^1 & & \vtx{1} \ar@/^/[ll]^u \ar@/^/[rr] & & \hdots \ar@/^/[ll]}$}
\item{$\xymatrix@=.5cm{\vtx{1} \ar@/^/[rr]^{\begin{bmatrix} 1 \\ u \end{bmatrix}} & & \vtx{2} \ar@/^/[ll]^{\begin{bmatrix} v & w \end{bmatrix}} \ar@/^/[rr] & & \hdots \ar@/^/[ll]}$}
\item{$\xymatrix@=.5cm{\vtx{2} \ar@/^/[rr]^{\begin{bmatrix} 1 & 0 \\ 0 & 1 \end{bmatrix}} & & \vtx{2} \ar@/^/[ll]^{\begin{bmatrix} u & v \\ w & x \end{bmatrix}} \ar@/^/[rr] & & \hdots \ar@/^/[ll]}$}
\end{itemize}
\end{lemma}

\begin{proof} Before we add an extra vertex, the family of representations is simple and hence the stabilizer subgroup at the end-vertex is reduced to $\C^*$. Adding the vertex and arrows, we can quotient out the base-change action at the new end-vertex by composing the two new arrows (the 'first fundamental theorem for $GL_n$-invariant theory, see \cite[Thm. II.4.1]{Kraftbook}). As a consequence, we have to classifying resp. a $1 \times 1$ matrix, a rank one $2 \times 2$ matrix or a $2 \times 2$ matrix with trivial action. The given representations do this.
\end{proof}

Assuming the total number of vertices is $\leq 5$ we can 'glue'  two such subquivers and families of representations at a common end-vertex having dimension $1$. Indeed, the ring of polynomial invariants of the glued quivers is the tensor product of those of the two subquivers as any oriented cycle passing through the glue-vertex can be decomposed into the product of two oriented cycles, one belonging to each component. When the two subquivers $Q$ and $Q'$ are glued along a $1$-vertex located at vertex $i$ we will denote the new subquiver $Q \bullet_i Q'$.
Remains only the problem of 'closing-the-circle' by adding the $6$-th vertex to one of these families of representations on the remaining five vertices.

\begin{lemma} \label{lem3}
Assume we have one of these dimension vectors on the full subquiver of $Q_{\mathcal{S}_1}$ on five vertices and the corresponding family of generically simple representations. Then, we can extend the dimension vector and the rational family of representations to the full quiver $Q_{\mathcal{S}_1}$
\begin{itemize}
\item{$\xymatrix@=.5cm{\hdots \ar@/^/[rr] & & \vtx{1} \ar@/^/[ll] \ar@/^/[rr]^{1} & & \vtx{1} \ar@/^/[ll]^u \ar@/^/[rr]^v & & \vtx{1} \ar@/^/[ll]^w \ar@/^/[rr] & & \hdots \ar@/^/[ll]}$}
\item{$\xymatrix@=.5cm{\hdots \ar@/^/[rr] & & \vtx{1} \ar@/^/[ll] \ar@/^/[rr]^{1} & & \vtx{1} \ar@/^/[ll]^u \ar@/^/[rr]^{\begin{bmatrix} v \\ w \end{bmatrix}} & & \vtx{2} \ar@/^/[ll]^{\begin{bmatrix} x & y \end{bmatrix}} \ar@/^/[rr] & & \hdots \ar@/^/[ll]}$}
\item{$\xymatrix@=.5cm{\hdots \ar@/^/[rr] & & \vtx{2} \ar@/^/[ll] \ar@/^/[rr]^{\begin{bmatrix} v & w \end{bmatrix}} & & \vtx{1} \ar@/^/[ll]^{\begin{bmatrix} 1 \\ u \end{bmatrix}} \ar@/^/[rr]^{\begin{bmatrix} x \\ y \end{bmatrix}} & & \vtx{2} \ar@/^/[ll]^{\begin{bmatrix} p & q \end{bmatrix}} \ar@/^/[rr] & & \hdots \ar@/^/[ll]}$}
\end{itemize}
\end{lemma}

\begin{proof} Forgetting the right-hand arrows, the left-hand representations are added as in the previous lemma. Then, the action on the right-hand arrows is trivial.
\end{proof}

For any of the obtained $Q_{\mathcal{S}_1}$-dimension vectors, we can now describe a rational family of generically simple representations via a code containing the following ingredients 
\begin{itemize}
\item{$A$,$B$,$C$,$D$,$E$ will be representations as in lemma~\ref{lem1}, $\overline{C}$ and $\overline{E}$ will denote the mirror images of $C$ and $E$}
\item{we add $1$'s or $2$'s when we add these dimensions to the appropriate side as in lemma~\ref{lem2}}
\item{we add $\bullet_i$ when we glue along a $1$-vertex placed at spot $i$, and}
\item{we add $1_j$ if we close-up the family at a $1$-vertex placed at spot $j$ as in lemma~\ref{lem3}}
\end{itemize}

\begin{theorem} In Figure~1 we list rational families for all irreducible components of $\wis{iss}_n~B_3$ for dimensions $n < 12$. A $+$-sign in column two indicates there are two such components, mirroring each other, $\beta_M$ indicates the $Q_0$-dimension vector and $\alpha_M$ gives the multiplicities of the $\mathcal{S}_1$-simples of a semi-simple $\Gamma_0$-representation in the component. The last column gives the dimension of this $\wis{iss}_n~B_3$-component.
\end{theorem}

\begin{figure} \label{ratfams}
\[
\begin{array}{|c|c|l|l|l|c|}
\hline
type & & \beta_M & \alpha_M & code & dim \\
\hline
& & & \\
6a & + & (3,3;3,2,1) & (1,1,1,2,1,0) & 11B & 6 \\
6b &    & (3,3;2,2,2) & (1,1,1,1,1,1) & 1_1A111 & 8 \\
6c &    & (4,2;2,2,2) & (1,1,2,1,1,0) & 1B1 & 6 \\
\hline
& & & \\
7a & + & (4,3;3,3,1) & (2,2,1,1,1,0) & \overline{C}11 & 7 \\
7b &    & (4,3;3,2,2) & (2,1,1,1,1,1) & 1_41B1 & 9 \\
\hline
& & & \\
8a &   & (4,4;4,2,2) & (2,2,2,2,0,0) & D2 & 10 \\
8b & + & (4,4;4,3,1) & (2,2,1,2,1,0) & \overline{C} \bullet_3 B & 8 \\
8c & + & (4,4;3,3,2) & (2,2,1,1,1,1) & 1_5 1 \overline{C} 1& 12  \\
8d & + & (5,3;3,3,2) & (2,1,1,1,2,1) &  1_3 B \bullet_6 B & 10 \\
\hline
& & & \\
9a & + & (5,4;4,4,1) & (2,2,1,2,2,0) & \overline{C} \bullet_3 C & 9 \\
9b & + & (5,4;4,3,2) & (2,1,1,2,2,1) & 1_3 \overline{C} \bullet_6 B & 13  \\
9c &    & (5,4;3,3,3) & (2,2,2,1,1,1) & 1_5 1 D 1 & 15 \\
9d &    & (6,3;3,3,3) & (3,2,2,0,1,1) &  1 E 2 & 11 \\
\hline
& & & \\
10a & + & (5,5;5,4,1) & (2,2,1,3,2,0) & \overline{C} \bullet_3 \overline{F} & 10 \\
10b & + & (5,5;5,3,2) & (3,2,2,2,1,1) & 1_5 E 2 2 & 14 \\
10c & + & (5,5;4,4,2) & (2,2,1,2,2,1)& 1_3 \overline{C} \bullet_6 C  & 16 \\
10d &    & (5,5;4,3,3) & (3,2,2,1,1,1) & 1_5 E 21 & 18 \\
10e & + & (6,4;4,4,2) & (2,2,2,2,2,0) & D22 & 14 \\
10f &     & (6,4;4,3,3) & (3,2,2,1,1,1) & 1_4 1 E 2 & 16 \\
\hline
& & & \\
11a & + & (6,5;5,5,1) & (3,2,0,2,3,1) & \overline{E} \bullet_6 E & 11 \\
11b & + & (6,5;5,4,2) & (3,2,1,2,2,1) & 1_3 \overline{C} \bullet_6 E & 17  \\
11c &     & (6,5;5,3,3) & (3,2,2,2,1,1) & 1_5 E 2 2 & 19 \\
11d & + & (6,5;4,4,3) & (2,2,2,2,2,1) & 1_6 D 2 2 & 21 \\
11e & + & (7,4;4,4,3) & (3,2,2,1,2,1) & 2_3 B \bullet_6 E & 17 \\
\hline
\end{array}
\]
\caption{Rational dense families of $B_3$-representations}
\end{figure}

\begin{proof} We are looking for $Q_0$-dimension vectors $\beta=(a,b;x,y,z)$ satisfying $a+b=x+y+z=n$, $max(x,y,z) \leq min(a,b)$ and $a \geq b$, $x \geq y \geq z$ (if $y \not= z$ there is a mirror component). One verifies that one only obtains the given 5-tuples and that the indicated $Q_{\mathcal{S}_1}$-dimension vectors $\alpha_M$ are compatible with $\beta$ and that the code is allowed by the previous lemmas. Only in the final entry, type $11e$ we are forced to close-up in a 2-vertex, but by the argument given in lemma~\ref{lem3} it follows that in this case the
following representations
\[ 
\xymatrix{\hdots \ar@/^/[rr] & & \vtx{2} \ar@/^/[ll] \ar@/^/[rr]^{\begin{bmatrix} 1 & 0 \\ 0 & 1 \end{bmatrix}} & & \vtx{2} \ar@/^/[ll]^{\begin{bmatrix} u & v \\ w & x \end{bmatrix}} \ar@/^/[rr]^{\begin{bmatrix} y & z \end{bmatrix}} & & \vtx{1} \ar@/^/[ll]^{\begin{bmatrix} p \\ q \end{bmatrix}} \ar@/^/[rr] & & \hdots \ar@/^/[ll]}
\]
will extend the family to a rational dense family in $\wis{iss}(Q_{\mathcal{S}_1},\alpha_M)$.
\end{proof}

For example, the code $1_3 \overline{C} \bullet_6 C$ of component $10c$ will denote the following rational family of $Q_{\mathcal{S}_1}$-representations
\[
\xymatrix@=1.8cm{
& \vtx{2} \ar@/^/[ld]|{\begin{bmatrix} 1 & a \end{bmatrix}} \ar@/^/[rd]^{\begin{bmatrix} 1 & 0 \\ 0 & 1 \end{bmatrix}} & \\
\vtx{1} \ar@/^/[ru]^{\begin{bmatrix} 0 \\ 1 \end{bmatrix}}  \ar@/^/[d]^{\begin{bmatrix} 0 \\ 1 \end{bmatrix}} & & \vtx{2} \ar@/^/[lu]^{\begin{bmatrix} b & c \\ d & 1 \end{bmatrix}} \ar@/^/[d]^{\begin{bmatrix} i & j \end{bmatrix}} \\
\vtx{2} \ar@/^/[u]^{\begin{bmatrix} 1 & e \end{bmatrix}}  \ar@/^/[rd]^{\begin{bmatrix} 1 & 0 \\ 0 & 1 \end{bmatrix}} & & \vtx{1} \ar@/^/[u]^{\begin{bmatrix} 1 \\ k \end{bmatrix}} \ar@/^/[ld]^{\begin{bmatrix} l \\ m \end{bmatrix}} \\
& \vtx{2} \ar@/^/[lu]^{\begin{bmatrix} f & g \\ h & 1 \end{bmatrix}} \ar@/^/[ru]|{\begin{bmatrix} n & p \end{bmatrix}}  & }
\]
Corresponding to these representations are the $B$-matrices of theorem~\ref{matrices}
\[
B = \begin{bmatrix}
\begin{array}{cc|c|cc|cc|cc|c}
1 & 0 & 0 & 0 & 0 & b & c & 0 & 0 & 0 \\
0 & 1 & 0 & 0 & 0 & d & 1 & 0 & 0 & 1 \\
\hline
0 & 0 & l & 1 & 0 & 0 & 0 & 1 & 0 & 0 \\
0 & 0 & m & 0 & 1 & 0 & 0 & 0 & 1 & 0 \\
\hline
1 & 0 & 1 & 0 & 0 & 1 & 0 & 0 & 0 & 0 \\
0 & 1 & k & 0 & 0 & 0 & 1 & 0 & 0 & 0 \\
\hline
0 & 0 & 0 & 1 & 0 & 0 & 0 & f & g & 0 \\
0 & 0 & 0 & 0 & 1 & 0 & 0 & h & 1 & 1 \\
\hline
0 & 0 & 1 & 0 & 0 & i & j & n & p & 0 \\
\hline
1 & a & 0 & 1 & e & 0 & 0 & 0 & 0 & 1
\end{array}
\end{bmatrix}
\]
which give us a rational $16=15+1$-dimensional family of $B_3$-representations, Zariski dense in their component.

\section{Detecting knot invertibility}

The rational Zariski dense families of $B_3$-representations can be used to separate conjugacy classes of three string braids by taking traces. Recall that a {\em flype} is a braid of the form
\[
b = \sigma_1^u \sigma_2^v \sigma_1^w \sigma_2^{\epsilon} \]
where $u,v,w \in \Z$ and $\epsilon = \pm 1$. A flype is said to be non-degenerate if $b$ and the reversed braid $b' = \sigma_2^{\epsilon} \sigma_1^w \sigma_2^v \sigma_1^u$ are in distinct conjugacy classes.  An example of a non-degenerate flype of minimal length is $b=\sigma_1^{-1} \sigma_2^2 \sigma_1^{-1} \sigma_2$
\[
\includegraphics{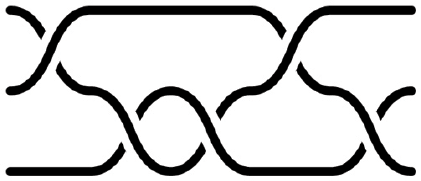} \]
Hence we can ask whether a $6$-dimensional $B_3$-representation can detect that $b$ lies in another conjugacy class than its reversed braid $b' = \sigma_2 \sigma_1^{-1} \sigma_2^2 \sigma_1^{-1}$
\[
\includegraphics{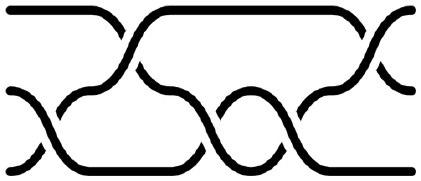} \]
Note that from the classification of all simple $B_3$-representation of dimension $\leq 5$ by Tuba and Wenzl \cite{TubaWenzl} no such representation can separate $b$ from $b'$. Let us consider $6$-dimensional $B_3$-representations. From the previous section we retain that there are $4$ irreducible components in $\wis{iss}_6~B_3$ (two being mirror images of each other) and that rational parametrizations for the corresponding $\Gamma_0$-components are given by the following representations of type 6a,6b and 6c
\[
\xymatrix@=1.1cm{
& \vtx{1}  \ar@/^/[rd]^{1} & \\
   & & \vtx{1} \ar@/^/[lu]^{a} \ar@/^/[d]^{b} \\
\vtx{1}   \ar@/^/[rd]|{{\tiny \begin{bmatrix} 0 \\ 1 \end{bmatrix}}} & & \vtx{1} \ar@/^/[u]^{1} \ar@/^/[ld]^{{\tiny \begin{bmatrix} 1 \\ 0 \end{bmatrix}}} \\
& \vtx{2} \ar@/^/[lu]^{{\tiny \begin{bmatrix} 1 & e \end{bmatrix}}} \ar@/^/[ru]|{{\tiny \begin{bmatrix} c & d \end{bmatrix}}}  & }~\qquad~\xymatrix@=1.1cm{
& \vtx{1} \ar@/^/[ld]^{g} \ar@/^/[rd]^{1} & \\
\vtx{1} \ar@/^/[ru]^{f}  \ar@/^/[d]^{e} & & \vtx{1} \ar@/^/[lu]^{a} \ar@/^/[d]^{b} \\
\vtx{1} \ar@/^/[u]^{1}  \ar@/^/[rd]^{1} & & \vtx{1} \ar@/^/[u]^{1} \ar@/^/[ld]^{1} \\
& \vtx{1} \ar@/^/[lu]^{d} \ar@/^/[ru]^{c}  & }~\qquad~
\xymatrix@=1.1cm{
& \vtx{1}  \ar@/^/[rd]^{1} & \\
    & & \vtx{1} \ar@/^/[lu]^{a} \ar@/^/[d]^{{\tiny \begin{bmatrix} 1 \\ 0 \end{bmatrix}}} \\
\vtx{1}   \ar@/^/[rd]^{1} & & \vtx{2} \ar@/^/[u]^{{\tiny \begin{bmatrix} c & d \end{bmatrix}}} \ar@/^/[ld]^{{\tiny \begin{bmatrix} 1 & e \end{bmatrix}}} \\
& \vtx{1} \ar@/^/[lu]^{b} \ar@/^/[ru]|{{\tiny \begin{bmatrix} 0 \\ 1 \end{bmatrix}}}  & }\]
As a consequence we obtain a rational Zariski dense subset of $\wis{iss}_6~B_3$ from theorem~\ref{matrices} using the following matrices $B$ for the three components
\[
\begin{bmatrix} 1 & 0 & 0 & a & 0 & 0 \\
0 & 1 & 0 & 0 & 1 & 0 \\
0 & 0 & 1 & 0 & 0 & 1 \\
1 & 1 & 0 & 1 & 0 & 0 \\
0 & 0 & 1 & 0 & 1 & e \\
0 & 1 & 0 & b & c & d \end{bmatrix}~\qquad~
\begin{bmatrix} 1 & 0 & 0 & a & 0 & f \\
0 & 1 & 1 & 0 & 1 & 0 \\
1 & 1 & 0 & 1 & 0 & 0 \\
0 & 0 & 1 & 0 & d & e \\
0 & 1 & 0 & b & c & 0 \\
g & 0 & 1 & 0 & 0 & 1 \end{bmatrix}~\qquad~\begin{bmatrix} 1 & 0 & 0 & 0 & a & 0 \\
0 & 1 & e & 1 & 0 & 1 \\
1 & c & d & 0 & 1 & 0 \\
0 & 0 & 0 & 1 & 0 & b \\
0 & 1 & 0 & 0 & 1 & 0 \\
0 & 0 & 1 & 0 & 0 & 1 \end{bmatrix}
\]
Using a computer algebra system, for example SAGE \cite{Sage},  one verifies that $Tr(b)=Tr(b')$ for all $6$-dimensional $B_3$-representations belonging to components $6a$ and $6c$. However, $Tr(b) \not= Tr(b')$ for an open subset of representations in component $6b$ and hence $6$-dimensional $B_3$-representations can detect non-degeneracy of flypes. The specific representation given in the introduction (obtained by specializing $a=c=e=g=1$ and $b=d=f=\lambda=-1$ gives $Tr(b)=648 \rho -228$ whereas $Tr(b')=-648 \rho -876$.
In order to use the family of $6$-dimensional $B_3$-representations as an efficient test to separate 3-braids from their reversed braids it is best to specialize the variables to random integers in $\Z[\rho]$ to obtain an $8$-parameter family of $B_3$-representations over $\mathbb{Q}(\rho)$. In SAGE \cite{Sage} one can do this as follows

\begin{verbatim}

K.<l>=NumberField(x^2+x+1);
a=randint(1,1000)*l+randint(1,1000);
b=randint(1,1000)*l+randint(1,1000);
c=randint(1,1000)*l+randint(1,1000);
d=randint(1,1000)*l+randint(1,1000);
e=randint(1,1000)*l+randint(1,1000);
f=randint(1,1000)*l+randint(1,1000);
g=randint(1,1000)*l+randint(1,1000);
h=randint(1,1000)*l+randint(1,1000);
B=matrix(K,[[1,0,0,a,0,f],[0,1,1,0,1,0],[1,1,0,1,0,0],[0,0,1,0,d,e],
[0,1,0,b,c,0],[g,0,1,0,0,1]]);
Binv=B.inverse();
mat2=matrix(K,[[1,0,0,0,0,0],[0,1,0,0,0,0],[0,0,1,0,0,0],[0,0,0,-1,0,0],
[0,0,0,0,-1,0],[0,0,0,0,0,-1]]);
mat3=matrix(K,[[1,0,0,0,0,0],[0,1,0,0,0,0],[0,0,l^2,0,0,0],
[0,0,0,l^2,0,0],[0,0,0,0,l,0],[0,0,0,0,0,l]]);
s1=h*Binv*mat3*B*mat2;
s2=h*mat2*Binv*mat3*B;
s1inv=s1.inverse();
s2inv=s2.inverse();

\end{verbatim}

One can then test the ability of this family to separate braids from their reversed braid on the list of all knots having at most $8$ crossings and being closures of three string braids, as provided by the Knot Atlas \cite{KnotAtlas}.  
It turns out that the braids of the following knots can be separated from their reversed braids : $6_3$, $7_5$, $8_7$, $8_9$, $8_{10}$ (all flypes), as well as the smallest non-invertible knot $8_{17}$ as mentioned in the introduction.

\end{document}